\documentclass[a4paper,abstracton]{scrartcl}

\usepackage[utf8]{inputenc}

\usepackage{amsmath}
\usepackage{amsthm}
\usepackage{amssymb}

\usepackage{algorithm,algorithmic}

\usepackage[usenames,dvipsnames]{color}
\usepackage{graphicx}
\usepackage{subfigure}
\usepackage{multirow}
\usepackage{todonotes}
\usepackage{xspace}

\usepackage{amssymb,amsmath,verbatim,tabularx,enumitem,colonequals,graphicx,psfrag,subfigure}

\usepackage{cite}
\usepackage{url}

\newcommand{\cU}{\mathcal{U}}

\newcommand{\R}{\mathbb{R}}
\newcommand{\X}{\mathcal{X}}

\newcommand{\rmp}{RMP\xspace}

\newcommand{\selection}{\textsc{Selection}\xspace}
\newcommand{\tsp}{\textsc{TSP}\xspace}

\newcommand{\mro}{\textsc{MRO}\xspace}
\newcommand{\midpt}{\textsc{MID}\xspace}

\usepackage{authblk}
\usepackage[normalem]{ulem} 

\usepackage{framed}
\usepackage{rotating}

\newtheorem{theorem}{Theorem}

\newtheorem{assumption}[theorem]{Assumption}

\newcommand{\algru}{RU\xspace}
\newcommand{\algex}{MRO-Ex\xspace}
\newcommand{\algmid}{Mid\xspace}

\newcommand{\algheu}{MRO-Heu\xspace}
\newcommand{\algldr}{MRO-LDR\xspace}
\newcommand{\algcg}{MRO-CG\xspace}
\newcommand{\algls}{MRO-LSHeu\xspace}

\newcommand{\hiro}{\textsc{Hiro}\xspace}

\newcommand{\BIGOP}[1]{\mathop{\mathchoice%
{\raise-0.22em\hbox{\huge $#1$}}%
{\raise-0.05em\hbox{\Large $#1$}}{\hbox{\large $#1$}}{#1}}}

\allowdisplaybreaks

\begin{document}

\title{Generating Hard Instances for Robust Combinatorial Optimization}
\author[1]{Marc Goerigk\thanks{Corresponding author. Email: marc.goerigk@uni-siegen.de}}
\affil[1]{Network and Data Science Management, University of Siegen, Germany}

\author[2]{Stephen J. Maher}
\affil[2]{Department of Management Science, Lancaster University, United Kingdom}

\date{}

\maketitle

\abstract{While research in robust optimization has attracted considerable interest over the last decades, its algorithmic development has been hindered by several factors. One of them is a missing set of benchmark instances that make algorithm performance better comparable, and makes reproducing instances unnecessary. Such a benchmark set should contain hard instances in particular, but so far, the standard approach to produce instances has been to sample values randomly from a uniform distribution.

In this paper we introduce a new method to produce hard instances for min-max combinatorial optimization problems, which is based on an optimization model itself. Our approach does not make any assumptions on the problem structure and can thus be applied to any combinatorial problem. Using the \selection and \textsc{Traveling Salesman} problems as examples, we show that it is possible to produce instances which are up to 500 times harder to solve for a mixed-integer programming solver than the current state-of-the-art instances.}


\section{Introduction}

We consider (nominal) combinatorial optimization problems of the form
\[ \min_{\pmb{x}\in\X} \pmb{c}\pmb{x} \]
where $\X\subseteq\{0,1\}^n$ denotes the set of feasible solutions, and $\pmb{c}\in\R^n_+$ is a cost vector. For the case that the cost coefficients $\pmb{c}$ are not known exactly, robust optimization approaches have been developed. In the most basic form, we assume a discrete set $\cU=\{\pmb{c}^1,\ldots,\pmb{c}^N\}$ of possible costs to be given, the so-called uncertainty set. Depending on the problem application, $\cU$ may be found by sampling from a distribution, or by using past observations of data. The robust (min-max) problem is then to solve
\[ \min_{\pmb{x}\in\X} \max_{\pmb{c}\in\cU} \pmb{c} \pmb{x} \]
This type of problem was first introduced in \cite{KouYu97}, and several surveys are now available, see \cite{Aissi2009,goerigk2016algorithm,kasperski2016robust}. The robust problem turns out to be NP-hard for all relevant problems that have been considered so far, even for $N=2$. This is also the case if the nominal problem is solvable in polynomial time, for example the \textsc{Shortest Path} or the \textsc{Assignment} problem \cite{kasperski2016robust}.

However, practical experience tells us that an NP-hard problem can sometimes still be solved sufficiently fast for relevant problem sizes. In fact, where NP-hardness proofs typically rely on constructing problem instances with specific properties, nothing is known about hardness of randomly generated instances, or smoothed analysis, in robust optimization. Where the related min-max regret problem has sparked research into specialized solution algorithms (see, e.g., \cite{catanzaro2011reduction,pereira2011exact,kasperski2012tabu}), little such research exists for the min-max problem, as simply using an off-the-shelf mixed-integer programming solver, such as CPLEX, can already lead to satisfactory results.

Faced with a similar situation for nominal knapsack problems, \cite{pisinger2005hard} asked: ``Where are the hard knapsack problems?'' The related aim of this paper is to construct computationally challenging robust optimization problems. To this end, we consider the \selection problem, where
$ \X = \left\{\pmb{x}\in\{0,1\}^n : \sum_{k=1}^n x_k = p \right\}$,
and the \textsc{Traveling Salesman} problem (\tsp) as examples. The nominal problem of the former can be solved in polynomial time, while it is NP-hard for the latter.
However, the proposed methods are general and can be applied to any robust combinatorial problem. 

Looking into other fields of optimization problems, instance libraries have been a main driver of algorithm development \cite{muller2010algorithm}. Examples include MIPLIB~\cite{koch2011miplib} for mixed-integer programs, road networks from the DIMACS challenge for \textsc{Shortest Path} problems \cite{demetrescu2009shortest} or the Solomon instances for the vehicle routing problem with time windows\cite{solomon1987}. There is a clear gap in robust optimization, where instance generators often need to be re-implemented to reproduce previous results. Our research is intended as a first step towards a library of hard instances to guide future research. Both our benchmark set of instances and the code to generate them are published on a website dedicated to this purpose, \url{www.robust-optimization.com}.

As there is no free lunch in optimization, we cannot hope to construct instances that are hard for all possible optimization algorithms. We therefore avoid constructing instances that are hard for a particular solution method (e.g., using CPLEX), but rather aim at maximizing hypothetical measures of hardness. Whether or not they actually correspond to harder instances for the solver is then a matter of computational experiments.

Our focus is to find an uncertainty set such that the optimal objective value of the resulting robust problem is as large as possible (Section~\ref{sec:1}). To solve the resulting optimization problem, Section~\ref{sec:solutionApproaches} considers several exact and heuristic solution methods. We briefly discuss our software package for instance generation in Section~\ref{sec:software}, before comparing solution approaches and the hardness of the resulting instances in Section~\ref{sec:exp}. We find that it is possible to construct instances that are considerably harder to solve than i.i.d. uniformly sampled problems---the current standard approach. Section~\ref{sec:conclusions} concludes the paper and points out further avenues for research.

\section{An Optimization Model for Maximizing the Robust Objective Value}\label{sec:1}

This paper proposes the use of an optimization problem to construct hard problem instances.
Throughout this section the proposed model is presented along with a number of different solution techniques.
In the presentation of the model and related discussions, the vectors and matrices are written in bold font, for example $\pmb{x} = (x_1,\ldots,x_n)$, and for sets $\{1,\ldots,n\}$ the shorthand notation $[n]$ is used.

Let some problem instance with $N$ scenarios be given, represented through the scenario objective coefficient vectors $\tilde{\pmb{c}}^1,\ldots,\tilde{\pmb{c}}^N$, with $\tilde{\pmb{c}}^i\in\R^n_+$. From this initial instance, the goal is to modify the inputs in such a way that the resulting robust problem is harder to solve. The approach that is used in this paper is to modify the values of the cost vectors in each of the scenarios. However, the base problem is to be modified, and not completely changed, so a limit on the magnitude of the change for each cost value is imposed.

Consider a scenario $i\in[N]$, which is a vector of cost coefficients denoted by $\tilde{\pmb{c}}^{i}$. The modification of the problem involves the selection of cost coefficients from the set of all possible candidate values, which is denoted by $\cU_i$. In the approach proposed in this paper, the set $\cU_{i}$ is defined as
\[ \cU_i = \left\{ \pmb{c}\in\R^n_+ : c_k\in[\underline{c}^i_k,\overline{c}^i_k]\  \forall k\in[n], \sum_{k\in[n]} c_k = \sum_{k\in[n]} \tilde{c}^i_k \right\} \]
where $\underline{c}^i_k$ and $\overline{c}^i_k$ denote the lower and upper bounds, respectively, on the cost coefficient $k$. Additionally, $\cU_{i}$ imposes the constraint that the sum of coefficients for this scenario remains the same, but any feasible sum that respects the upper and lower bounds is permitted as a scenario vector. We will use $\underline{c}^i_k = \max\{\tilde{c}^i_k-b,0\}$ and $\overline{c}^i_k = \min\{\tilde{c}^i_k+b,C\}$ with a budget parameter $b$ and a global maximum cost coefficient $C$.

Our approach aims at finding scenarios $\pmb{c}^i\in\cU_i$ for all $i\in[N]$, so that the objective value of the optimal solution to the corresponding robust optimization problem is increased. This approach can be formulated as the following optimization problem
\begin{equation}
  \max_{\pmb{c}^i\in\cU_i \forall i\in[N]} \min_{\pmb{x}\in\X} \max_{j\in[N]} \pmb{c}^j\pmb{x} \tag{\mro}
  \label{eqn:mro}
\end{equation}
where \mro stands for ``maximize robust objective''. 
The intuition behind the proposed optimization problem for generating difficult robust problem instances is the following: For each $\pmb{x}\in\X$, the objective $\max_{j\in[N]} \pmb{c}^j \pmb{x}$ is a piecewise linear, convex function in $\pmb{c}^1,\ldots,\pmb{c}^N$. By maximizing the smallest value of the objective over all $\pmb{x}$, we spread out the solution costs, balancing the objective values of the best solutions in $\X$. This way, finding and proving optimality of the best $\pmb{x}$ becomes a more difficult task for an optimization algorithm. Naturally, whether the instances produced using the proposed method are actually more difficult to solve than the original problem $\tilde{\pmb{c}}^1,\ldots,\tilde{\pmb{c}}^N$ can only be tested computationally.

As an example, consider a robust variant of the \selection problem where the task is to choose two out of four items such that the maximum costs over two scenarios are as small as possible. The cost vectors for these two scenarios are
\begin{center}
\begin{tabular}{r|rrrr}
 & \multicolumn{4}{c}{item} \\
 & 1 & 2 & 3 & 4\\
\hline
$\pmb{c}^1$ & 4 & 1 & 9 & 2 \\
$\pmb{c}^2$ & 4 & 7 & 4 & 4
\end{tabular}
\end{center}
In this small example there are $\binom{4}{2} = 6$ possible solutions. For this particular instance of the robust \textsc{selection} problem there is only one optimal solution to this problem, which is to choose items 1 and 4 with a robust objective value 8. The sorted vector of the corresponding six robust objective values is
\[ (8,11,11,11,11,13) \]
Now let us assume that $\pmb{c}^1 \in \cU_{1}$ and $\pmb{c}^2 \in \cU_{2}$ and the budget is given by $b=1$. Thus, two alternative cost vectors $\hat{\pmb{c}}^1 \in \cU_{1}$ and $\hat{\pmb{c}}^2 \in \cU_{2}$ are
\begin{center}
\begin{tabular}{r|rrrr}
 & \multicolumn{4}{c}{item} \\
 & 1 & 2 & 3 & 4\\
\hline
$\hat{\pmb{c}}^1$ & 3 & 2 & 10 & 1 \\
$\hat{\pmb{c}}^2$ & 5 & 6 & 3 & 5
\end{tabular}
\end{center}
Given these cost vectors, the objective value of the optimal robust solution increases to 10. The optimal solution still remains as the selection of 1 and 4, but the sorted vector of robust objective values has become
\[ (10,11,11,11,12,13) \]
An important observation is that the difference between the best and second-best solutions has reduced. This can have the effect of increasing the difficulty of proving optimality. As mentioned previously, the difficulty of the instance can only be evaluated computationally. Using CPLEX to solve the min-max robust \textsc{selection} problem given by this small example, the first instance takes 0.013 ticks of the deterministic clock, whereas the second instances is solved in 0.209 deterministic ticks---more than 16 times as long.

\section{Solution Approaches}
\label{sec:solutionApproaches}

A clear drawback of \mro is that the inner problem is the robust optimization problem that we are attempting to make hard. Therefore, constructing a hard problem is at least as hard as actually solving it. Due to this fact we primarily focus on producing hard, but relatively small instances. This is an alternative to the trivial approach to producing hard instances, which is to produce larger ones. In the last part of this section we shift our focus to the generation of large and hard robust optimization instances. This is achieved through the use of heuristic methods to solve the \mro.

Note that even evaluating the objective value of some fixed scenario variables $\pmb{c}^1,\ldots,\pmb{c}^N$ is NP-hard for all commonly considered combinatorial problems (see \cite{kasperski2016robust}), as they are equivalent to solving a robust counterpart. Indeed, in Appendix~\ref{app:complexity} we show that \mro is $\Sigma^p_2$-complete when scenarios can be chose from polyhedral sets.

In the outer maximization problem, we determine $N$ vectors, and choose one of these vectors in the inner maximization problem. Formally, this is similar to the $K$-adaptability approach in robust optimization (see \cite{buchheim2017min}), which uses a min-max-min structure. Whereas their combinatorial part is in the outer minimization, the combinatorial part is in the inner minimization in our problem.

To address the difficulty of \mro, different solution approaches are developed. Each of the solution approaches aim to reduce the difficulty of solving \mro through alternative techniques. These approaches are:
\begin{itemize}
  \item Iterative method (Section~\ref{sec:it}): an exact approach that exploits the multi-level structure of \mro.
  \item Column generation method (Section~\ref{sec:cg}): an exact approach that applies decomposition to a relaxation of \mro.
  \item Alternating heuristic (Section~\ref{sec:alternate}): a heuristic applied to a reformulation of \mro.
  \item Linear decision rules (Section~\ref{sec:ldr}): a heuristic method to find a compact formulation of \mro.
  \item Replacing the inner subproblem with a heuristic (Section~\ref{sec:large}): a method for larger \mro problems.
\end{itemize}
A description of each of the solution approaches is presented in the following sections. The experimental results in Section \ref{sec:exp} then demonstrate the value of each approach.

\subsection{Iterative Solution}\label{sec:it}

Given the multi-level structure, it is difficult to solve \mro directly using general purpose solvers. However, decomposition techniques can be used to exploit this structure and to develop an effective solution approach.

Note that we can write the inner maximization problem for given $\pmb{c}^i$, $i\in[N]$, and $\pmb{x}\in\X$ by introducing a variable vector $\pmb{\lambda}$ representing the choice of scenario:
\[ \max_{j\in[N]} \pmb{c}^j\pmb{x} = \max \left\{ \sum_{i\in[N]} \lambda_i \pmb{c}^i \pmb{x} : \sum_{i\in[N]} \lambda_i = 1, \ \lambda_i \in \{0,1\}\ \forall i\in[N] \right\} \]

Let us now assume that some set $\{\pmb{x}^1,\ldots,\pmb{x}^K\}\subseteq\X$ of candidate solutions are already known. Then, the restricted \mro problem on this set can be written as
\begin{equation}
  \begin{aligned}
  \max\ & t \\
  \text{s.t. } & t \le (\sum_{i\in[N]} \lambda^j_i \pmb{c}^i) \pmb{x}^j & \forall j\in[K] \\
  & \sum_{i\in[N]} \lambda^j_i = 1 & \forall j\in[K] \\
  & \pmb{c}^i \in \cU_i & \forall i\in[N] \\
  & \lambda^j_i \in \{0,1\} & \forall i\in[N], j\in[K]
  \end{aligned}
  \label{eqn:reformulatedMRO}
\end{equation}
where the variables $\pmb{\lambda}^j$ for each $j\in[K]$ are used to determine the scenario $\pmb{c}$ that is assigned to each candidate $\pmb{x}^j$. We refer to this problem also as the master problem.

Note that problem \eqref{eqn:reformulatedMRO} is nonlinear through the product of $\pmb{\lambda}$ and $\pmb{c}$ variables, which can be linearized using variables $d_{ijk} = \lambda^j_i c^i_k$. The resulting model is then given as
\begin{equation}
  \begin{aligned}
  \max\ & t \\
  \text{s.t. } & t \le \sum_{i\in[N]}\sum_{k\in[n]} d_{ijk} x^j_k & \forall j\in[K] \\
  & \sum_{i\in[N]} \lambda^j_i = 1 & \forall j\in[K] \\
  & d_{ijk} \le c^i_k & \forall i\in[N],j\in[K],k\in[n]\\
  & d_{ijk} \le \overline{c}^i_k\lambda^j_i & \forall i\in[N],j\in[K],k\in[n] \\
  & \pmb{c}^i \in \cU_i & \forall i\in[N] \\
  & \lambda^j_i \in \{0,1\} & \forall i\in[N], j\in[K]
  \end{aligned}
  \label{eqn:linearisedReformulatedMRO}
\end{equation}

Once the master problem is solved for a fixed set of candidate solutions, we have determined an upper bound on the \mro problem. By solving the resulting robust optimization problem for $\pmb{x}$, we also construct a lower bound. If both bounds are not equal, we add the current robust solution $\pmb{x}$ to the set of candidate solutions and repeat the process by solving the master problem. This iterative approach will converge after a finite number of steps, as $\X$ contains a finite number of solutions. It is therefore an exact solution approach to \mro.

An interesting question is whether the master problem is solvable in polynomial time. Note that for $N$ scenarios and $K$ solutions, there are $N^K$ possibilities to assign solutions to scenarios. For each assignment, constructing optimal scenarios $\pmb{c}$ can be done in polynomial time by solving a linear program. This means that if $K$ is constant, the master problem can be solved in polynomial time as well.

If $K$ is unbounded, however, the problem becomes hard, as the following theorem shows.

\begin{theorem}
The master problem is NP-hard, if $K$ is part of the input.
\end{theorem}
\begin{proof}
We use a reduction from \textsc{Hitting Set}, see \cite{ga1979computers}: Given a ground set $[E]$, a collection of sets $S_1,\ldots,S_T\subseteq[E]$, and some integer $L\le E$. Is there a subset $C\subseteq[E]$ with $|C|\le L$ such that $|C|\cap S_i\neq \emptyset$ for all $i\in[T]$?

Let an instance of \textsc{Hitting Set} be given. We set $n=E$, $N=L$ and $K=T$. We further set $b=C=1$, and $\tilde{c}_k = 1/n$ for each $k\in[n]$ (i.e., we get $\underline{c}_k = 0$ and $\overline{c}_k=1$ for all $k\in[n]$). Finally, we set $x^i_k = 1$ if $k\in S_i$ and $x^i_k=0$ otherwise.

We now claim that \textsc{Hitting Set} is a yes-instance if and only if there is a solution to \mro with objective value at least 1.

To prove this claim, let us first assume \textsc{Hitting Set} is a yes-instance. Let $C=\{e_1,\ldots,e_L\}$ be a corresponding subset of $[E]$ (w.l.o.g. we assume that $|C|=L$). Then we build a solution to \mro in the following way. For each $e_\ell\in C$, set $c^\ell_{e_\ell} = 1$ and $c^\ell_k = 0$ for all $k\neq e_\ell$. For each $S_i$, choose one $e_\ell \in S_i \cap C$ and set $\lambda^i_{\ell} = 1$ and all other $\lambda^i_k = 0$. Thus we obtain a feasible solution to \mro with objective value at least 1.

We illustrate this process with a small example. Let $E=\{1,\ldots,7\}$, $S_1 = \{1,2,3\}$, $S_2=\{3,4,5\}$, $S_3=\{6,7\}$, and $L=2$. Our \mro instance has the following values of $\pmb{x}^1,\pmb{x}^2,\pmb{x}^3$ and $\tilde{\pmb{c}}^1$ and $\tilde{\pmb{c}}^2$:
\begin{center}
\begin{tabular}{l|rrrrrrr}
$\pmb{x}^1$ & 1 & 1 & 1 & 0 & 0 & 0 & 0 \\
$\pmb{x}^2$ & 0 & 0 & 1 & 1 & 1 & 0 & 0 \\
$\pmb{x}^3$ & 0 & 0 & 0 & 0 & 0 & 1 & 1 \\
\hline
$\tilde{\pmb{c}}^1$ & $1/7$ & $1/7$ & $1/7$ & $1/7$ & $1/7$ & $1/7$ & $1/7$ \\
$\tilde{\pmb{c}}^2$ & $1/7$ & $1/7$ & $1/7$ & $1/7$ & $1/7$ & $1/7$ & $1/7$ \\
\hline
$\pmb{c}^1$ & 0 & 0 & 1 & 0 & 0 & 0 & 0 \\
$\pmb{c}^2$ & 0 & 0 & 0 & 0 & 0 & 1 & 0 
\end{tabular}
\end{center}
In the same table, we also show an optimal solution for $\pmb{c}^1$ and $\pmb{c}^2$. The $\pmb{\lambda}$ variables are chosen such that $\pmb{x}^1$ and $\pmb{x}^2$ are assigned to $\pmb{c}^1$, and $\pmb{x}^3$ is assigned to $\pmb{c}^2$.

Now let us assume that for some \textsc{Hitting Set} instance, we construct our \mro problem as detailed above and find an objective value of at least 1. We show that \textsc{Hitting Set} is a yes-instance. To this end, we first show that there exists an optimal solution to \mro where all $\pmb{c}^i_k$-variables are either 0 or 1. Consider any $\pmb{c}^i$, and let $\pmb{x}^{i_1},\ldots,\pmb{x}^{i_p}$ be all $\pmb{x}$-solutions assigned to scenario $i$. We distinguish two cases:
\begin{enumerate}
\item There exists some $s\in[n]$ such that $x^{i_j}_s=1$ for all $j\in[p]$. In this case, we can set $c^i_s = 1$.
\item There is no such $s\in[n]$, i.e., there are $x^{i_k}$ and $x^{i_\ell}$ with $k,\ell\in[p]$ that choose disjoint sets of items. As $\sum_{k\in[n]} c^i_k = 1$, at least one of them must have an objective value strictly less than 1, which contradicts our assumptions.
\end{enumerate}
We can thus set $C$ by including all elements $k\in[n]$ for which there is $i\in[N]$ with $c^i_k=1$. By construction, $C$ is a hitting set with cardinality at most $L$.

\end{proof}

While the iterative algorithm is an exact solution approach, there are limitations to its use. Specifically, solving the master problem can become a bottleneck to the solution process as the number of solutions $K$ increases. In each iteration of the algorithm, the addition of a new candidate $\pmb{x}$ results in an additional $2 + 2Nn$ constraints. Computationally, the additional constraints have a significant negative impact between consecutive iterations. Two different solution methods will be presented to address the issue in solving the master problem. A Dantzig-Wolfe decomposition approach will be presented in Section \ref{sec:cg} and an alternating heuristic will be described in Section \ref{sec:alternate}.

\subsection{Column Generation}\label{sec:cg}

Dantzig-Wolfe reformulation is applied to \eqref{eqn:linearisedReformulatedMRO} to decompose the problem into $K$ disjoint subsystems---one for each candidate solution $x$.
A column $p$ corresponds to a feasible assignment of a cost vector $\pmb{c}^{i} \in \cU_{i}$ to the solution vector $\pmb{x}^{j}$.
For a given column $p \in P^{j}$, the parameter $\overline{d}_{ikp}$ is the contribution of $c^{i}_{k}x^{j}_{k}$ to the objective of the inner minimization problem given the assignment of $\pmb{c}^{i}$ to solution vector $\pmb{x}^{j}$.
The variables $\alpha_{p}$ equal 1 if the cost vector assignment given by column $p$ is selected and 0 otherwise.
Finally, the variables $c^{i}_{k}$ are introduced to map the solution of the outer maximization problem to the set of cost vectors for the inner minimization problem.

The formulation of the column generation master problem is given by
\begin{equation}
  \begin{aligned}
  \max\ & t \\
  \text{s.t. } & t \le \sum_{p \in P^{j}}\sum_{i \in [N]}\sum_{k \in [n]}\overline{d}_{ikp}\alpha_{p} & \forall j \in [K] & \ (\gamma_{j}) \\
  & \sum_{p \in P^{j}} \sum_{p \in P^{j}}\alpha_{p} = 1 & \forall j \in [K] & \ (\delta_{j}) \\
  & \sum_{p \in P^{j}}\overline{d}_{ikp}\alpha_{p} \le c^{i}_{k} & \forall i \in [N], j \in [K], k \in [n] & \ (\pi_{ijk}) \\
  & \pmb{c}^i \in \cU_i & \forall i\in[N] \\
  & \alpha_{p} \in \mathbb{Z}_{+} & \forall j\in[K], p \in P^{j} &
  \end{aligned}
  \label{eqn:columnGenerationRMP}
\end{equation}

Initially, the master problem is formulated with only a subset of columns $\bar{P}^{j} \subseteq P^{j}$.
The corresponding problem is described as the restricted master problem (\rmp).
For each $j \in [K]$, a single initial column is included in $\bar{P}^{j}$, which is formed by assigning $\overline{c}_{k}$ to $x^{j}_{k}$.
The variables $\gamma_{j},\, \delta_{j}$ and $\pi_{ijk}$ represent the dual variables corresponding to the constraints in \eqref{eqn:columnGenerationRMP}.

A complicating aspect of the \rmp is the set of linking constraints given by the uncertainty sets $\cU_{i}$.
This complication arises from the fact that the constraints do not explicitly link the $\alpha_{p}$ variables, but an implicit linking of the $\alpha_{p}$ variables is through the third set of constraints in \eqref{eqn:columnGenerationRMP}.
While the uncertainty set linking constraints ensure that exactly one cost vector is selected from each scenario, this requirement could be overly restrictive in our contexts.
As such, a relaxation of \eqref{eqn:linearisedReformulatedMRO} is formed by replacing $\cU_{i}$ with $\cU^{j}_{i}$, where $\cU^{j}_{i} = \cU_{i}, \forall j \in [K]$, so that a different cost vector from scenario $i$ could be selected for each solution $j \in [K]$.
Applying this relaxation eliminates the linking constraints from the uncertainty sets $\cU_{i}$ and transfers the additional relaxed constraints to the column generation subproblems.

A column generation subproblem is formed for each solution $j \in [K]$.
Given the optimal dual solution to the \rmp, each column generation subproblem is solved to find a feasible cost vector assignment that has a positive reduced cost.
The dual variables are denoted by $\gamma_{j},\, \delta_{j}$ and $\pi_{ijk}$ respectively for the constraints of the \rmp.
Using an optimal dual solution---denoted by $(\overline{\gamma}_{j}, \overline{\delta}_{j}, \overline{\pmb{\pi}}_{j})$---the reduced cost of a column for solution $j$ is given by
\begin{equation}
  \overline{d}^{j} = \sum_{i \in [N]}\sum_{k \in [n]}d_{ijk} x^j_k \overline{\gamma}_{j} - d_{ijk} \overline{\pi}_{ijk} - \overline{\delta}_{j}
  \label{eqn:reducedCost}
\end{equation}
A feasible assignment of $\pmb{c}^{i} \in \cU^{j}_{i}$ to solution $\pmb{x}^{j}$ forms an improving column for the \rmp if \eqref{eqn:reducedCost} is positive.
The feasible cost vector assignment that forms a column with the most positive reduced cost is found by solving the subproblem given by
\begin{equation}
  \begin{aligned}
  \hat{d}^{j} = \max\ & \sum_{i \in [N]}\sum_{k \in [n]}d_{ik} x^j_k \overline{\gamma}_{j} - d_{ik} \overline{\pi}_{ijk} - \overline{\delta}_{j}\\
  \text{s.t. } & \sum_{i\in[N]} \lambda_i = 1 \\
  & d_{ik} \le c^i_k & \forall i\in[N],k\in[n]\\
  & d_{ik} \le \overline{c}_k\lambda_i & \forall i\in[N],k\in[n] \\
  & \pmb{c}^i \in \cU^{j}_{i} & \forall i\in[N] \\
  & \lambda_i \in \{0,1\} & \forall i\in[N]
  \end{aligned}  
\end{equation}

The optimal solution to \eqref{eqn:columnGenerationRMP} provides a scenario set that is expected to form a \emph{hard} robust optimization problem.
Since only a relaxation of \eqref{eqn:linearisedReformulatedMRO} is solved by this approach, objective function value will be greater than that found by the iterative approach (Section \ref{sec:it}).
However, in the proposed approach for generating hard instances, maximizing the minimum robust objective value is used only as a proxy for hardness.
As such, it is expected that even solving the relaxation of \eqref{eqn:linearisedReformulatedMRO} will provide instances that are of comparative hardness to the exact approach in Section \ref{sec:it}.

\subsection{Alternating Heuristic}\label{sec:alternate}

As an alternative to the relaxation and decomposition approach presented in Section \ref{sec:cg}, an alternating heuristic has been developed to solve the master problem \eqref{eqn:linearisedReformulatedMRO} of the iterative approach.
The alternating heuristic is motivated by the observation that for a given assignment of scenarios $i \in [N]$ to solutions $j \in [K]$, selecting the cost coefficients to maximize the minimum objective becomes a simple task.
Similarly, for a fixed set of cost coefficients for each scenario, the difficulty in assigning scenarios to solutions is greatly reduced.
As such, the alternating heuristic iterates between fixing either the scenario assignment or the scenario cost coefficients.

To formally present the alternating heuristic, first reconsider the master problem from the iterative method
\begin{align*}
\max\ & t \\
\text{s.t. } & t \le (\sum_{i\in[N]} \lambda^j_i \pmb{c}^i) \pmb{x}^j & \forall j\in[K] \\
& \sum_{i\in[N]} \lambda^j_i = 1 & \forall j\in[K] \\
& \pmb{c}^i \in \cU_i & \forall i\in[N] \\
& \lambda^j_i \in \{0,1\} & \forall i\in[N], j\in[K]
\end{align*}
for a subset $\{\pmb{x}^1,\ldots,\pmb{x}^K\}\subseteq\X$ of solutions. Let us assume the variables $\pmb{c}^i$ are all fixed. In that case, an optimal solution to the remaining $\pmb{\lambda}$ variables can be found through the following procedure: For each $j\in[K]$, choose one $i\in[N]$ such that $\pmb{c}^i\pmb{x}^j$ is not smaller than $\pmb{c}^\ell\pmb{x}^j$ for all $\ell\neq i$. Then, set $\lambda^j_i=1$ and all other $\lambda^j_\ell = 0$. To determine which $\pmb{c}^i$ is a maximizer of the objective value for some $\pmb{x}^j$, we can simply calculate all $N$ possible objective values. Thus, finding optimal $\pmb{\lambda}$ values is possible in $\mathcal{O}(nNK)$ time. Now let us assume that all $\pmb{\lambda}$ variables are fixed. In this case, the remaining variables are continuous. Under the assumption that the $\cU_i$ are polyhedra, the resulting problem can then be solved in polynomial time as well. This leads to the alternating heuristic described in Algorithm \ref{alg:alternate}.

\begin{algorithm}[h]
  \caption{Alternating heuristic to solve \eqref{eqn:reformulatedMRO}}
  \label{alg:alternate}
  \begin{algorithmic}[1]
    \REQUIRE set of solutions $\pmb{x}$, set of scenarios $\cU_{i}$, an initial set of cost coefficients $\tilde{\pmb{c}}$
    \ENSURE Cost coefficients $\hat{\pmb{c}}_{i}$ for each scenario $i \in [N]$
    \STATE Let $\hat{\pmb{c}} \leftarrow \tilde{\pmb{c}}$, $t \leftarrow -1$, $z_{t} \leftarrow 0$
    \REPEAT
      \STATE $t \leftarrow t + 1$
      \STATE fix the value of $\pmb{c}$ to $\hat{\pmb{c}}$ and solve \eqref{eqn:reformulatedMRO} for $\pmb{\lambda}$
      \STATE set $\hat{\pmb{\lambda}} \leftarrow \pmb{\lambda}$
      \STATE fix the value of $\pmb{\lambda}$ to $\hat{\pmb{\lambda}}$ and solve \eqref{eqn:reformulatedMRO} for $\pmb{c}$
      \STATE set $\hat{\pmb{c}} \leftarrow \pmb{c}$
      \STATE set $z_{t}$ to the current objective value of \eqref{eqn:reformulatedMRO}
    \UNTIL{$z_{t - 1} \geq z_{t}$}
  \end{algorithmic}
\end{algorithm}

\subsection{Linear Decision Rules}\label{sec:ldr}

A common reformulation of robust optimization problems involves the application of decision rules\cite{ben2004adjustable}.
This approach 
involves introducing the adjustable variables $\lambda^i: \{0,1\}^n \to [0,1]$ which map solutions $\pmb{x}$ to the worst-case scenario.
In the context of \mro, such a mapping would result in setting $\lambda^i(\pmb{x})=1$ if scenario $\pmb{c}^i$ is a worst-case scenario for solution $\pmb{x}$, and 0 otherwise.

Considering the \mro, the use of a decision rule results in an equivalent formulation given by
\begin{equation}
  \begin{aligned}
    \max\ & t\\
    \text{s.t. } & t \le \sum_{k\in[n]} \sum_{i\in[N]} \lambda^i(\pmb{x})c^i_k x_k & \forall \pmb{x}\in\X \\
    & \sum_{i\in[N]} \lambda^i(\pmb{x}) \le 1 & \forall \pmb{x}\in\X\\
    & \lambda^i: \{0,1\}^n \to [0,1] & \forall i\in[N] \\
    & \pmb{c}^i \in \cU_i & \forall i\in[N]
  \end{aligned}
  \label{eqn:mroOptimalDecisionRule}
\end{equation}

The optimal decision rule can only be found through the solution to the original robust optimization problem.
As such, it is common to apply approximations of the decision rules to find a closed form of the reformulated problem.
First-order or linear decision rules involve defining the vector mapping $\lambda^i(\pmb{x})$ as an affine linear function, such as
\[ \lambda^i(\pmb{x}) := \lambda^i_0 + \sum_{k\in[n]} \lambda^i_k x_k. \]
This introduces the new variables $\lambda^i_0$, $\lambda^i_k$ for all $k\in[n]$.
An approximation of \mro is given by substituting the linear function mapping in \eqref{eqn:mroOptimalDecisionRule}, resulting in the reformulation given by
\begin{align}
\max\ & t \label{ldrReformulationObj}\\
\text{s.t. } & t \le \sum_{k\in[n]} \sum_{i\in[N]} (\lambda^i_0 + \sum_{\ell\in[n]} \lambda^i_\ell x_\ell) c^i_k x_k & \forall \pmb{x}\in\X \label{con1}\\
& \sum_{i\in[N]} (\lambda^i_0 + \sum_{k\in[n]} \lambda^i_k x_k) \le 1 &\forall \pmb{x}\in\X  \label{con2}\\
& \lambda^i_0 + \sum_{k\in[n]} \lambda^i_k x_k \ge 0 & \forall i\in[N],\pmb{x}\in\X  \label{con3}\\
& \lambda^i_0 + \sum_{k\in[n]} \lambda^i_k x_k \le 1 & \forall i\in[N],\pmb{x}\in\X  \label{con4}\\
& \pmb{c}^i \in \cU_i & \forall i\in[N] \label{con5}
\end{align}
Note that it is possible to remove constraints \eqref{con4} since they are implied by constraints~\eqref{con2} and \eqref{con3}.

It can be observed that the reformulated problem has an exponential number of constraints, resulting from a set of constraints for each solution contained in $\X$.
As such, problem \eqref{ldrReformulationObj}--\eqref{con5} is intractable in its current form.
Using the following linear relaxation assumption,
a further reformulation can be performed to address the intractability of problem \eqref{ldrReformulationObj}--\eqref{con5}

\begin{assumption}\label{lpassumption}
There exists a suitable polyhedron
\[ \X' = \{ \pmb{x}\in\R^n_+ : \pmb{A}\pmb{x} \le \pmb{b}, \pmb{x} \le \pmb{1} \} \]
with $\pmb{A}\in\R^{m\times n}$, $\pmb{b}\in\R^m$, such that 
for any cost vector $\pmb{c}\in\R^n$, we have
\[ \min_{\pmb{x}\in\X} \pmb{c}\pmb{x} = \min_{\pmb{x}\in\X'} \pmb{c}\pmb{x}. \]
\end{assumption}

To apply Assumption~\ref{lpassumption}, each set of constraints in \eqref{con1}--\eqref{con3} are examined in turn to construct a polyhedral description of linear constraints.
For each set of constraints, the bounding limit is found by minimizing (maximizing) the activity for greater (less) than constraints.
Assumption~\ref{lpassumption} is given for a wide range of commonly considered robust combinatorial optimization problems, such as \selection, \textsc{Spanning Tree}, and \textsc{Assignment} (see also \cite{kasperski2016robust}). 

Consider any constraint of the form
\begin{equation}\label{exeq}
\pmb{c}\pmb{x} \le C \qquad \forall \pmb{x}\in\X
\end{equation}
for some vector $\pmb{c}$ and right-hand side $C$. 
This is equivalent to
\begin{equation}\label{consMax}
\max_{\pmb{x}\in\X} \pmb{c}\pmb{x} \le C.
\end{equation}
Using strong duality, this means that
\begin{equation}\label{consDual}
\min_{\pmb{u}\in D(\X')} \pmb{b}\pmb{u} \le C
\end{equation}
where $D(\X')$ denotes the set of dual feasible solutions and $\pmb{b}$ is the right-hand (left-hand) side of less (greater) than constraints in \eqref{consMax}. Due to weak duality, it is further sufficient to find some $\pmb{u}'\in D(\X')$ such that $\pmb{b}\pmb{u}' \le C$, which implies that Inequality~\eqref{exeq} is fulfilled. Analogously, for constraints of the form
\[\pmb{c}\pmb{x} \ge C \qquad \forall \pmb{x}\in\X\]
the original mathematical program that is dualized in the reformulation is
\[\min_{\pmb{x} \in \X}\pmb{c}\pmb{x} \ge C.\]
This reformulation approach is applied to Constraints~\eqref{con1}--\eqref{con3} to form a tractable problem.

For ease of presentation, we describe the reformulation 
using \selection as an example.
Consider Constraint~\eqref{con1}, which is equivalent to
\[ t \le \min_{\pmb{x}\in\X} \sum_{k\in[n]} (\sum_{i\in[N]} \lambda^i_0 c^i_k) x_k + \sum_{k\in[n]} \sum_{\ell\in[n]} (\sum_{i\in[N]} \lambda^i_\ell c^i_k) x_\ell x_k \]
First the product $x_\ell x_k$ is linearized by introducing a new variable $y_{kl}$. Then the resulting problem can be relaxed to form
\begin{align*}
\min\ & \sum_{k\in[n]} (\sum_{i\in[N]} \lambda^i_0c^i_k ) x_k + \sum_{\ell\in[n]} (\sum_{i\in[N]} \lambda^i_\ell c^i_k) y_{k\ell} \\
& 2 y_{k\ell} \le x_\ell + x_k & \forall \ell,k\in[n] \\
& x_k + x_\ell \le y_{k\ell} + 1 & \forall \ell,k\in[n] \\
& y_{k\ell} \ge 0 & \forall \ell,k\in[n] \\
& \sum_{i\in[n]} x_i = p \\
& x_k \le 1 & \forall k\in[n] \\
& x_k \ge 0 & \forall k\in[n] 
\end{align*}
Note that this will give a conservative approximation to Constraint~\eqref{con1}, as the minimum in the right-hand side is underestimated. Also, the right-hand side of \eqref{con1} is ignored when applying Assumption \ref{lpassumption}, since it will be enforced in the reformulation of \mro. By dualizing the problem, we find
\begin{align*}
\max\ & -\sum_{k\in[n]} \sum_{\ell\in[n]} \zeta_{k\ell} + p\eta - \sum_{k\in[n]} \theta_k \\
\text{s.t. } & \sum_{\ell\in[n]} (\xi_{k\ell} + \xi_{\ell k} - \zeta_{k\ell} - \zeta_{\ell k}) + \eta - \theta_k \le \sum_{i\in[N]} \lambda^i_0c^i_k & \forall k\in[n] \\
& -2\xi_{k\ell} + \zeta_{k\ell} \le \sum_{i\in[N]} \lambda^i_\ell c^i_k & \forall k,\ell\in[n] \\
& \xi,\zeta,\theta \ge 0 \\
& \eta \gtrless 0
\end{align*}
By strong duality, this model can be substituted for Constraint~\eqref{con1}.

Consider Constraint~\eqref{con2}:
\[ \max_{\pmb{x}\in\X} \sum_{i\in[N]} (\lambda^i_0 + \sum_{k\in[n]} \lambda^i_k x_k) \le 1 \]
The linear programming reformulation of this constraint is given by
\begin{align*}
\max \ & \sum_{i\in[N]} \lambda^i_0 + \sum_{k\in[n]} (\sum_{i\in[N]}  \lambda^i_k ) x_k \\
\text{s.t. } & \sum_{k\in[n]} x_k = p \\
& x_k \le 1 & \forall k\in[n] \\
& x_k \ge 0 & \forall k\in[n] 
\end{align*}
Same as for Constraint \eqref{con1}, the right-hand side is ignored when applying Assumption \ref{lpassumption}. The dual of this problem is given by
\begin{align*}
\min\ & \sum_{i\in[N]} \lambda^i_0 + p \pi + \sum_{k\in[n]} \rho_k \\
\text{s.t. } & \pi + \rho_k \ge \sum_{i\in[N]}  \lambda^i_k & \forall k\in[n]\\
& \pi \gtrless 0 \\
& \rho_k \ge 0 & \forall k\in[n]
\end{align*}

Finally, we use the same approach for Constraint~\eqref{con3}. For each $i\in[N]$, we have
\[ \min_{\pmb{x}\in\X} \lambda^i_0 + \sum_{k\in[n]} \lambda^i_k x_k  \ge 0 \]
for which the dual
is
\begin{align*}
\max\ & \lambda^i_0 + p\alpha^i - \sum_{k\in[n]} \beta^i_k \\
\text{s.t. } & \alpha^i - \beta^i_k \le \lambda^i_k & \forall k\in[n] \\
& \alpha^i \gtrless 0 \\
& \beta^i_k \ge 0 & \forall k\in[n]
\end{align*}

Putting the above discussion together, the linear decision rule approach to \mro is given through the following optimization problem:
\begin{align}
  \max\ & p\eta -\sum_{k\in[n]} \sum_{\ell\in[n]} \zeta_{k\ell}  - \sum_{k\in[n]} \theta_k \label{ldr:obj}\\
  \text{s.t. } & \sum_{\ell\in[n]} (\xi_{k\ell} + \xi_{\ell k} - \zeta_{k\ell} - \zeta_{\ell k})\nonumber\\
  &\qquad\qquad+ \eta - \theta_k \le \sum_{i\in[N]} \lambda^i_0 c^i_k & \forall k\in[n] \label{ldr:con1}\\
  & -2\xi_{k\ell} + \zeta_{k\ell} \le \sum_{i\in[N]} \lambda^i_\ell c^i_k & \forall k,\ell\in[n] \label{ldr:con2}\\
  & \sum_{i\in[N]} \lambda^i_0 + p \pi + \sum_{k\in[n]} \rho_k \le 1 \label{ldr:con3}\\
  & \pi + \rho_k \ge \sum_{i\in[N]}  \lambda^i_k & \forall k\in[n]\label{ldr:con4}\\
  & \lambda^i_0 + p\alpha^i - \sum_{k\in[n]} \beta^i_k \ge 0 & \forall i\in[N]\label{ldr:con5}\\
  & \alpha^i - \beta^i_k \le \lambda^i_k & \forall k\in[n],i\in[N] \label{ldr:con6}\\
  & \xi,\zeta,\theta \ge 0 \label{ldr:con7}\\
  & \eta \gtrless 0 \label{ldr:con8}\\
  & \pi \gtrless 0 \label{ldr:con9}\\
  & \rho_k \ge 0 & \forall k\in[n] \label{ldr:con10}\\
  & \alpha^i \gtrless 0 & \forall i\in[N]\label{ldr:con11}\\
  & \beta^i_k \ge 0 & \forall k\in[n],i\in[N] \label{ldr:con12}\\
  & \lambda^i_k \gtrless 0 & \forall i\in[N], k\in[n]\cup\{0\} \label{ldr:con13}\\
  &\pmb{c}^i \in\cU_i & \forall i\in[N] \label{ldr:con14}
\end{align}
The
reformulation of constraint \eqref{con1} is given by the objective \eqref{ldr:obj} and constraints \eqref{ldr:con1}--\eqref{ldr:con2}. For constraint \eqref{con2}, the reformulation is given by \eqref{ldr:con3}--\eqref{ldr:con4}. Note that the right-hand side of \eqref{con2} is the right-hand side of \eqref{ldr:con3}. Finally, the reformulation of \eqref{con3} is given by \eqref{ldr:con5}--\eqref{ldr:con6}. Similarly, the left-hand side of \eqref{con3} is the left-hand side of \eqref{ldr:con5}.

There is still a nonlinearity between variables $\lambda^i_\ell$ and $c^i_k$, with $\lambda^i_\ell$ being unbounded. We solve the optimization problem heuristically, using an alternating approach similar to Section~\ref{sec:alternate}. By fixing either variables $\pmb{\lambda}$, $\pi$, $\pmb{\rho}$, $\pmb{\alpha}$, $\pmb{\beta}$ or variables $\pmb{c}$, we increase the current objective value in each iteration, until a local optimum has been reached.

Note that while we described the reformulation for the special case of \selection, the same method can be used for any problem with Assumption~\ref{lpassumption}.

\subsection{Heuristics for Large Instances}
\label{sec:large}

A limitation of the previously presented approaches is that the exact solution of the resulting robust optimization subproblem is required. Therefore, it is not possible to find instances that take longer to solve than the generation time. This is especially true for the iterative methods, where it is likely that the robust optimization problem will be solved multiple times. While generating hard instances to small robust optimization problems is appropriate for most benchmarking purposes, this limitation prohibits the generation of hard large instances.

Fortunately, with only a small modification the \mro is still possible to generate hard instances to large robust optimization problems.
Instead of solving the \mro
\[ \max_{\pmb{c}^i\in\cU_i \forall i\in[N]} \min_{\pmb{x}\in\X} \max_{j\in[N]} \pmb{c}^j\pmb{x}  \]
exactly, we can use any heuristic to solve the robust problem $\min_{\pmb{x}\in\X} \max_{j\in[N]} \pmb{c}^j\pmb{x}$. Given the nature of the \mro and the iterative solution approaches, any approximation of heuristic approach for finding a solution that is a proxy for the optimum---such as solving the linear relaxation of the robust problem and then rounding the solution---can be used. The solution to the robust optimization problem generated by the heuristic algorithm need not belong to $\X$.

Algorithmically, we can use the iterative method, the column generation approach, and the alternating heuristic in combination with the heuristic by just replacing any robust optimization subproblem.

\section{Software} \label{sec:software}

The approaches developed in Section \ref{sec:solutionApproaches} have been implemented within the {\hiro} (Hard Instances for Robust Optimization) C++ software library.
This library has been designed to facilitate the generation of hard instances to any robust optimization problem.

The \hiro software library provides a virtual function, \texttt{solve\_ip()}, within the \texttt{HIRO} class so that the user can specify the solution methods of an inner optimization problem when creating a problem specific derived class.
In the examples presented in this paper, the inner optimization problems of \selection and \tsp have been defined as integer programs that are solved using CPLEX.
However, it is also possible to use combinatorial or heuristic algorithms to solve the inner problem.
The parameters of \texttt{solve\_ip()} are the number of elements in the cost vector(such as the number of items in \selection or the number of edges in the \tsp), the number of scenarios and a two-dimensional vector containing the cost vectors for each scenario.
The provided set of cost vectors define the current hard robust optimization instance.
After solving the robust optimization problem, the user must create a \texttt{HIROSolution} object that is returned to the \hiro core.
The \texttt{HIROSolution} object comprises a solution vector and the upper bound of that solution.

The \hiro software library is publicly available and can be found on GitHub\footnote{\url{https://github.com/}} in the repository \texttt{stephenjmaher/HIRO}.
The two optimization problems investigated in this paper, \selection and \tsp, are included as examples in the software repository.
Extensive instructions for adding new examples are also provided.

\section{Experiments}\label{sec:exp}

The ultimate goal of this paper is to develop general purpose methods for generating hard instances for min-max robust optimization problems. To this end, we proposed the \mro model and different solution methods.
While the theoretical basis of the proposed optimization problem is expected to produce min-max robust optimization instances harder than randomly generated instances, verifying this is only possibly through empirical studies.

In this paper, the hardness of an instance is primarily measured as the run time required to solve it using CPLEX.
Due to the empirical nature of this work, it is not immediately obvious what the most effective method for generating hard instances is.
Further, it is not possible to determine a priori which of the proposed algorithms will produce the most difficult instance.
As such, it is necessary to evaluate the performance of the algorithms presented in Section~\ref{sec:solutionApproaches} 
with respect to instance generation time and generated instance hardness.

\subsection{Setup}

The approaches presented in Section~\ref{sec:solutionApproaches} are general methods that can be applied to the generation of hard instances for any min-max robust optimization problem.
The alternative methods that have been developed focus specifically on the computation of cost coefficients for each scenario, which is the master problem in the proposed iterative methods.
As such, the inner problem---the subproblem---can be set to any min-max robust optimization problem.
To demonstrate the potential of the methods from Section~\ref{sec:solutionApproaches}, robust variants of the \selection problem and \tsp are used as the inner problem.
The \selection problem is used for its simplicity, meaning that the impact of the instance generation is more easily observed.
We also consider the \tsp to demonstrate the potential of using \mro to generate hard instances for a computationally more challenging optimization problem.
We used the standard subtour elimination formulation for the \tsp, with lazy generation of subtour elimination constraints.

The current state-of-the-art for robust optimization instance generation is to randomly sample scenarios.
Thus, the baseline for comparison is a set of instances where the scenario coefficients are sampled randomly uniform i.i.d. with $c^i_k \in \{0,\ldots,100=C\}$.
This method of instance generation will be labeled as \algru.
In the following results, the proposed methods will be labeled as follows:
\begin{itemize}
\item \algex: The exact method from Section~\ref{sec:it}.
\item \algcg: The column generation method that is applied to the relaxation of \mro as described in Section~\ref{sec:cg}.
\item \algheu: The alternating heuristic from Section~\ref{sec:alternate}.
\item \algldr: The approach from Section~\ref{sec:ldr} where linear functions are used to approximate the assignment of solutions to scenarios.
\item \algls: The approach from Section~\ref{sec:large}, where we round the fractional solution found by solving the linear relaxation of the subproblem (in case of \tsp without using subtour elimination constraints).
\end{itemize}

We consider four different computational experiments: In the first, we use relatively small \selection problems, where the robust subproblem is solved to optimality (i.e., the methods from Sections~\ref{sec:it} to \ref{sec:ldr}). Three problem sizes are used: The number of items $n$ is set to $20$, $30$, and $40$. In each case we set the number of items to be selected $p$ to $n/2$ and the total number of scenarios $N$ is set to $n$. For each problem size, we generate 100 instances using \algru. The scenarios from these random instances are then used as the initial scenarios for the iterative methods described above. To evaluate the effect of the uncertainty set budget $b$ on the run times of the iterative methods and the hardness of the generated instances, budgets of 1, 2 and 5 are used. The total number of randomly generated instances is $3\cdot 100 = 300$. Since these instances are used as an input to each of the iterative methods, a further $3\cdot 3 \cdot 4 \cdot 100 = 3600$ \emph{hard} instances are generated. For the second experiment, we considered larger problems with $n$ ranging from 20 to 70, and a budget $b=20$. In this experiment we only compare \algru with \algls.

A maximum run time of 3600 seconds is given to each of the iterative methods. This run time limit is only enforced between the iterations of the algorithm, as such it is possible for this time limit to be exceeded. 

Experiments three and four repeat experiments one and two for the \tsp.
In this setting, $n$ is the number of edges in the complete graph, and hence $\sqrt{n}$ is the number of nodes.
We first produce random instances with $100$, $144$, and $196$ edges and in experiment three apply the \mro and associated solution approaches to generate hard instances for these smaller-scale problems.
In experiment four, large problem instances are considered with the number of edges, $n$, ranging from 100 to 400.
For both experiments three and four, the number of scenarios is set to $N=\sqrt{n}$.

The hardness of the instances is evaluated by using CPLEX to solve the resulting min-max robust optimization problem.
All experiments were conducted using one core of a computer with an Intel Xeon E5-2670 processor, running at 2.60 GHz with 20MB cache, with Ubuntu 12.04 and CPLEX v12.6.

\subsection{Hard Instance Generation for \selection}
\label{sec:hardInstancesSelection}

\begin{table}[htbp]
\begin{center}
\begin{tabular}{r|r|rrr}
& & \multicolumn{3}{c}{$n=N$} \\
Budget & Method & 20 & 30 & 40 \\
 \hline
\multirow{4}{*}{1} & \algex & 0.04 & 0.61 & 7.83 \\
 & \algcg & 0.04 & 0.45 & 6.79 \\
 & \algheu & 0.04 & 0.43 & 6.91 \\
 & \algldr & 0.03 & 0.17 & 2.20 \\
\hline
\multirow{4}{*}{2} & \algex & 0.06 & 0.92 & 7.49 \\
 & \algcg & 0.05 & 0.62 & 9.34 \\
 & \algheu & 0.06 & 0.85 & 15.36 \\
 & \algldr & 0.03 & 0.25 & 3.07 \\
\hline
\multirow{4}{*}{5} & \algex & 0.08 & 0.77 & 10.59 \\
 & \algcg & 0.06 & 0.78 & 11.92 \\
 & \algheu & 0.08 & 4.29 & 22.00 \\
 & \algldr & 0.04 & 0.46 & 8.21 \\
\hline
 & \algru & 0.03 & 0.13 & 1.43
\end{tabular}
\caption{\selection: Average CPU time in seconds when solving the generated problems.}\label{tab:cpu-solve}
\end{center}
\end{table}

The ability of the \mro to produce robust optimization instances harder than randomly generated instances for \selection is clearly demonstrated in Table \ref{tab:cpu-solve}.
Interestingly, as the budget $b$ increases, the dominance of \algheu emerges.
In fact, with $b=5$ and $n=40$ \algheu produces instances that are more than twice as hard, on average, as those produced by \algex.

\begin{table}[htb]
\begin{center}
\begin{tabular}{r|r|rrrrrr}
& & \multicolumn{6}{c}{$n=N$} \\
Budget & Method & \multicolumn{2}{c}{20} & \multicolumn{2}{c}{30} & \multicolumn{2}{c}{40} \\
\hline
\multirow{4}{*}{1} & \algex & 1.7 & (3.7) & 4.7 & (9.9) & 6.9 & (15.5) \\
& \algcg & 1.7 & (3.4) & 3.6 & (8.1) & 5.2 & (10.6) \\
& \algheu & 1.7 & (5.1) & 3.5 & (7.4) & 5.3 & (14.8) \\
& \algldr & 1.1 & (2.3) & 1.4 & (2.4) & 1.7 & (5.8) \\
\hline
\multirow{4}{*}{2} & \algex & 2.4 & (4.2) & 8.2 & (23.9) & 6.7 & (40.8) \\
& \algcg & 2.0 & (5.0) & 5.0 & (13.6) & 7.9 & (46.3) \\
& \algheu & 2.2 & (4.0) & 6.5 & (19.4) & 12.8 & (59.5) \\
& \algldr & 1.2 & (2.5) & 2.1 & (6.8) & 2.4 & (8.2) \\
\hline
\multirow{4}{*}{5} & \algex & 4.0 & (11.3) & 7.2 & (23.0) & 10.0 & (49.3) \\
& \algcg & 2.4 & (5.0) & 6.8 & (21.7) & 10.6 & (53.0) \\
& \algheu & 3.6 & (13.3) & 37.6 & (179.4) & 23.6 & (152.9) \\
& \algldr & 1.5 & (3.3) & 4.1 & (12.2) & 7.8 & (49.4)
\end{tabular}
\caption{\selection: Average (maximum) time increase relative to \algru for each set of instances.}\label{sel:timeinc}
\end{center}
\end{table}

The results from Table~\ref{tab:cpu-solve} are complemented by Table~\ref{sel:timeinc}, which shows that average and maximum increase in hardness of all instances (i.e., the ratio between the computation time of the hardened and the original instance). We observe that for $n=40$, method \algheu produces instances that take on average 23 times longer to solve, and up to over 150 times in some cases.

\begin{table}[htb]
\begin{center}
\begin{tabular}{r|r|rrr}
& & \multicolumn{3}{c}{$n=N$} \\
Budget & Method & 20 & 30 & 40 \\
 \hline
\multirow{4}{*}{1} & \algex & 1.7 & 207.0 & 3005.8 \\
 & \algcg & 1.0 & 16.1 & 253.4 \\
 & \algheu & 0.2 & 4.6 & 151.9 \\
 & \algldr & 2.1 & 31.3 & 232.1 \\
\hline
\multirow{4}{*}{2} & \algex & 49.3 & 3238.1 & 3796.0 \\
 & \algcg & 2.9 & 56.3 & 707.5 \\
 & \algheu & 0.6 & 24.2 & 947.9 \\
 & \algldr & 2.2 & 33.8 & 297.3 \\
\hline
\multirow{4}{*}{5} & \algex & 3677.9 & 4305.1 & 4081.9 \\
 & \algcg & 11.5 & 200.1 & 2035.6 \\
 & \algheu & 3.1 & 998.7 & 3614.6 \\
 & \algldr & 2.7 & 40.3 & 439.7 \\
\end{tabular}
\caption{\selection: Average CPU time in seconds to generate instances, using (lenient) 3600 seconds time limit.}\label{tab:cpu}
\end{center}
\end{table}

The inferior results of \algex are explained by Table~\ref{tab:cpu}, which shows a significant reduction in the run time for all methods compared to \algex. In most cases, \algex cannot complete within the lime limit.
The best performing approach when $b = 1$ is \algheu, with an average run time that is 5\% of that for \algex when $n = 40$.
This remains the case when $b = 5$, where \algldr outperforms all other methods.
While the decomposition approach \algcg does not dominate any of the other approaches, it exhibits its best performance compared to \algheu as $b$ and $n$ increases.

Comparing Table \ref{tab:cpu-solve} with the average generation times in Table \ref{tab:cpu}, it is observed that while \algldr is the fastest method for solving \mro, it is the worst method for producing hard robust optimization instances.
Also, the approximation methods of \algcg and \algheu, while faster than \algex, are able to produce hard robust optimization instances.
This demonstrates a clear advantage to the iterative methods of the \mro, in particular the use of relaxation and heuristic methods.

Considering the iterative methods, \algex, \algcg and \algheu, the exact approach generates the hardest instances when it is capable of solving \mro to optimality.
When \algex can not solve \mro to optimality, then \algheu produces the most difficult instances.
Interestingly, when $n = 40$ and $b = 2, 5$ the average run times of the instances generated by \algheu is twice as large as those generated by \algex.
While \algcg does not produce instances harder than \algheu, they are harder than those produced by \algex when \mro is not solved to optimality.

Figure~\ref{fig-a} gives a more detailed box plot comparison for the case $n=N=40$ and $b=5$, i.e., for the hardest instances. Note the logarithmic vertical axis. The highest observed computation time for {\algru} was 6.59 seconds, while {\algheu} achieved a maximum of 61.03 seconds. The Wilcoxon signed-rank test confirms that {\algheu} produces harder instances than {\algru} with $p<0.001$.

\begin{figure}[htbp]
\begin{center}
\includegraphics[width=.5\textwidth]{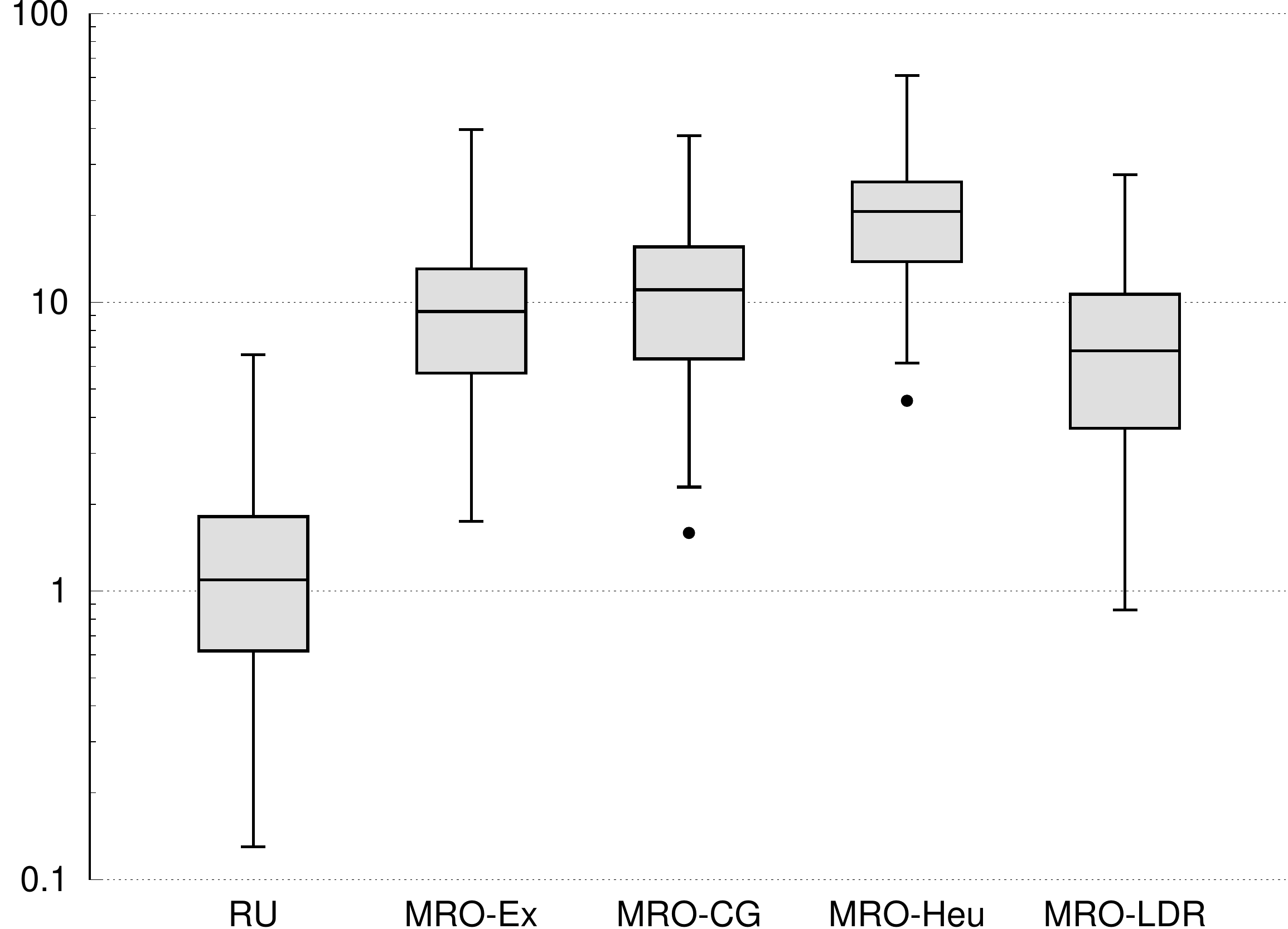}
\caption{\selection: Box plot of solution times (in seconds) for instances generated with $b=5$ and $n=N=40$.\label{fig-a}}
\end{center}
\end{figure}

\subsection{Large-Scale Instance Generation for \selection}

The ability of the \mro, especially the iterative methods, is demonstrated in Section \ref{sec:hardInstancesSelection} to produce instances significantly harder than randomly generated instances.
However, Table \ref{tab:cpu} shows that the generation time is a major limitation to the proposed approaches.
Specifically, as the instance size grows, the ability of \algex and \algheu to solve the \mro decreases.
Thus, it is necessary to employ alternative methods when generating hard robust optimization instances for larger problems.

Using the approach proposed in Section \ref{sec:large}---solving the inner robust optimization problem heuristically---hard instances to larger \selection problems can be generated.
In particular, we use a heuristic to solve the robust \selection problem within the alternating heuristic, labeled as \algls.
The results from using \algls, compared to \algru, for different numbers of items $n$ are presented in Table~\ref{tab:ls1}.
We show the average time for the instances that were solved to optimality and the number of instances which can be solved within the time limit of one hour.

\begin{table}[htb]
\begin{center}
\begin{tabular}{r|rrrr}
$n$ & \multicolumn{2}{c}{\algru} & \multicolumn{2}{c}{\algls} \\
\hline
20 & 0.02 & (100) & 0.06 & (100) \\
30 & 0.16 & (100) & 1.21 & (100) \\
40 & 1.72 & (100)& 25.45 & (100) \\
50 & 17.83 & (100) & 699.46 & (100) \\
60 & 163.65 & (100) & 1749.20 & (17) \\
70 & 1220.19 & (86) & - & (0)
\end{tabular}
\caption{\selection: Average CPU solving time in seconds for instances that were solved to optimality (number of instances that were solved to optimality).}\label{tab:ls1}
\end{center}
\end{table}

An initial observation from Table~\ref{tab:ls1} is that \algls is able to produce instances that are an order of magnitude harder than \algru, even for the smallest instances.
Comparing these results with Section \ref{sec:hardInstancesSelection}, it is clear that solving the robust optimization problem heuristically has advantages for the \mro as opposed to solving it exactly in \algheu.
Table \ref{tab:ls1} also demonstrates that as the size of the instances increases, many more of the instances produced by \algls cannot be solved within one hour decreases, compared to those produced by \algru.
Specifically, when $n = 60$, only 17 of the instances generated by \algls could be solved compared to 100 for the randomly generated instances.
This drops to 0 for the \algls instances when $n$ is increased to $70$.

\begin{figure}[htb]
\begin{center}
\includegraphics[width=.4\textwidth]{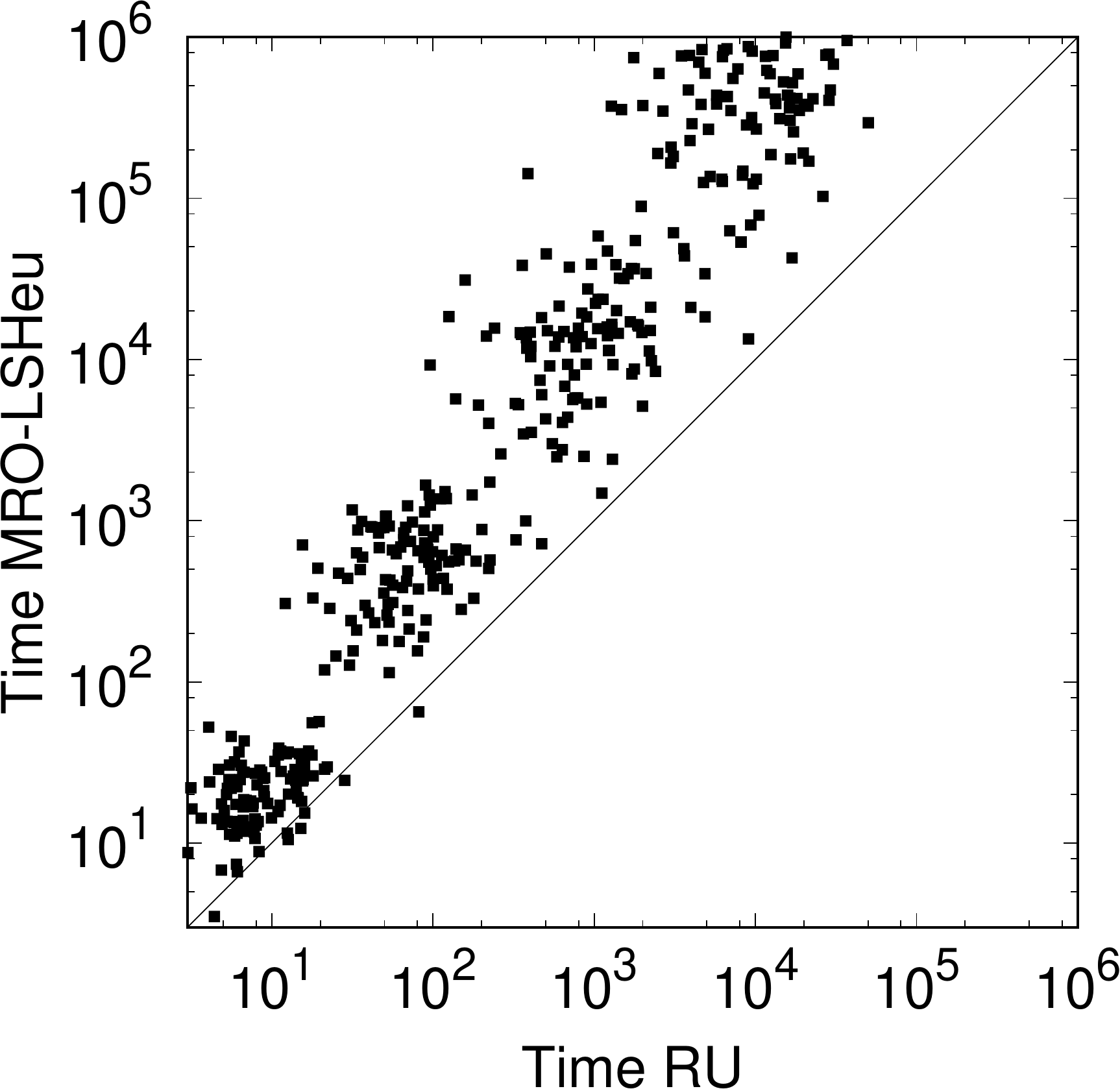}
\caption{\selection: Comparison of solution times per instance for $n=20$ to $n=50$.\label{fig-ls}}
\end{center}
\end{figure}

To better visualize this result, Figure~\ref{fig-ls} compares the run times required to solve the instances generated by \algru (horizontal axis) and \algls (vertical axis).
Note that in this figure the axes are scaled logarithmically.
Every point corresponds to one instance for $n=20$ to $n=50$ (problem sizes where all instances were solved to optimality).
A point on the diagonal means that a randomly generated instance remained as hard after applying \algls to increase its difficulty.
The clusters of dots are related to the values of $n$ and appear to display a linear relationship as $n$ increases.
Since this figure is displayed using double logarithmic axes, this linear relationship represents an exponential increase in the run times after applying \algls to produce hard instances.
Thus, \algls is capable of achieving significantly harder instances that \algru for large problem sizes of \selection.

We complement these observations with Table~\ref{tab:ls2}, where the average increase in solution time of \algls instances over the original \algru is shown. For $n=60$ and $n=70$, the true computation time is not always known. In these cases, we used 3600 seconds, but the actual increase will be higher than the numbers indicate. We see that with increasing instance size, our instance generation approach becomes more powerful, with an increase in hardness by a factor of up to 440.

\begin{table}[htbp]
\begin{center}
\begin{tabular}{r|rr}
$n$ & Avg & Max \\
\hline
20 & 2.8 & 12.9 \\
30 & 9.9 & 45.6 \\
40 & 23.1 & 195.9 \\
50 & 73.8 & 440.7 \\
60 & $>46.1$ & $>406.5$ \\
70 & $>4.75$ & $>44.3$
\end{tabular}
\caption{\selection: Average and maximum time increase for each instance using \algls. When not solved to optimality, counted as time limit.}\label{tab:ls2}
\end{center}
\end{table}

Finally, we also present the time to generate these instances in Table~\ref{tab:ls3}. On average, it takes less than 20 seconds to generate hard instances for $n=70$, which is a small fraction of the time needed to solve the resulting problems. This demonstrates that the proposed heuristic is an efficient tool to generate larger and harder instances, overcoming the limitations of the exact approach, and scaling well beyond $n=70$.

\begin{table}[htbp]
\begin{center}
\begin{tabular}{r|rrrrrr}
$n$ & 20 & 30 & 40 & 50 & 60 & 70 \\
\hline
Time & 0.1 & 0.3 & 0.9 & 2.0 & 4.4 & 19.4
\end{tabular}
\caption{\selection: Average time to generate instances using \algls.}\label{tab:ls3}
\end{center}
\end{table}

\subsection{Hard Instance Generation for \tsp}

We turn our attention to the \tsp for the next set of experimental results to assess whether the proposed methods can be applied to generate hard instances to more computationally difficult robust problems.
Given the dominance of the \algheu approach for the \selection problem, these experiments will only evaluate \algru, \algex and \algheu.
Also, for the larger problem sizes, the experiments will compare the difficulty of the instances generated by \algru and \algls.

The results of the first experiment, generating hard instances by \algex and \algheu with different values of $n=N^2$ and $b$, are presented in Table~\ref{tab:tsp1}.
Similar to the results in Section \ref{sec:hardInstancesSelection}, both \algex and \algheu produce instances that are significantly harder than \algru.
Also, of the considered algorithms \algheu is shown to generate the hardest instances.

\begin{table}[htbp]
\begin{center}
\begin{tabular}{r|r|rrr}
& & \multicolumn{3}{c}{$n=N^2$} \\
Budget & Method & 100 & 144 & 196 \\
\hline
\multirow{2}{*}{2} & \algex & 0.4 & 5.0 & 43.7 \\
 & \algheu & 0.5 & 6.7& 78.1 \\
\hline
 \multirow{2}{*}{5} & \algex & 1.4 & 10.7 & 52.2 \\
 & \algheu & 1.5 & 14.6  & 103.0 \\
 \hline & \algru & 0.2 & 1.1 & 9.0 \\
 \end{tabular}
\caption{\tsp: Average CPU time in second when solving the generated instances.}\label{tab:tsp1}
\end{center}
\end{table}

We find that with a larger budget, it is possible to increase the hardness considerably. While randomly generated instances for the 14-city \tsp take 9 seconds on average to solve, the instances hardened by using \algheu take on average 103 seconds to solve. The increase in hardness is further demonstrated by Table~\ref{tab:tsp2}, which shows the relative increase in computation time, averaged over all instances. We find the instances generated by \algheu exhibit a factor of 25 for $n=196$ longer than the randomly generated instances on average. Further, \algheu generates instances that take up to 230 times longer to solve than those originally produced by \algru.

\begin{table}[htbp]
\begin{center}
\begin{tabular}{r|r|rrrrrr}
& & \multicolumn{6}{c}{$n=N^2$} \\
Budget & Method & \multicolumn{2}{c}{100} & \multicolumn{2}{c}{144} & \multicolumn{2}{c}{196} \\
\hline
\multirow{2}{*}{2} & \algex & 2.8 & (5.3)  & 5.3 & (9.6) & 6.8 & (25.8)  \\
 & \algheu & 3.0 & (6.1) & 6.8 & (11.7) & 11.9 & (28.1) \\
\hline
 \multirow{2}{*}{5} & \algex & 10.0 & (24.1)  & 14.2 & (65.5) & 12.1 & (169.0) \\
 & \algheu & 11.0 & (27.2)  & 22.2 & (118.0)  & 25.0 & (232.7)
\end{tabular}
\caption{\tsp: Average (maximum) time increase relative to \algru for each set of instances.}\label{tab:tsp2}
\end{center}
\end{table}

\begin{table}[bp]
\begin{center}
\begin{tabular}{r|r|rrr}
& & \multicolumn{3}{c}{$n=N^2$} \\
 Budget & Method & 100 & 144 & 196 \\
\hline
\multirow{2}{*}{2} & \algex & 6.3  & 149.1 & 1327.9  \\
 & \algheu & 4.6  & 134.2 & 2028.4 \\
\hline
 \multirow{2}{*}{5} & \algex & 954.6  & 3170.0 & 3475.6 \\
 & \algheu & 143.1 & 2365.1  & 3507.0
\end{tabular}
\caption{\tsp: Average CPU time in seconds to generate instances, using (lenient) 3600 seconds time limit.}\label{tab:tsp3}
\end{center}
\end{table}

Similar to the results for \selection, the generation of hard instances for the \tsp is time consuming.
Table~\ref{tab:tsp3} shows that as budget $b$ and the number of edges $n$ increases, the run time required to generate hard instance also increases.
In fact, it is difficult to solve the \mro to optimality for any problems with more than 14 nodes in the underlying graph.
While the increase in difficulty for these smaller instances should be sufficient, if hard instances to larger problems are desired, then Table~\ref{tab:tsp3} highlights that \algex and \algheu are not satisfactory algorithms.
This shows the need for heuristic approaches for solving the inner robust optimization problem of the \mro, as proposed in Section \ref{sec:large}.

The ability of \algls to generate hard instances for larger problem sizes of the \tsp is shown in Table~\ref{tab:tsp4}.
In these experiments, instances were found for problems formulated with up to 20 nodes.
Also, in order to generate harder instances the budget $b$ was increased to 20.

\begin{table}[htbp]
\begin{center}
\begin{tabular}{r|rrrr}
$n$ & \multicolumn{2}{c}{\algru} & \multicolumn{2}{c}{\algls} \\
\hline
100 & 0.2 & (100) & 0.4 & (100)  \\
144 & 1.1 & (100) & 3.5 & (100) \\
196 & 9.3 & (100) & 66.0 & (100) \\
256 & 67.4 & (100) & 1040.7 & (99) \\
324 & 442.1 & (100) & 2512.8 & (2) \\
400 & 1624.2 & (60) & - & (0)
\end{tabular}
\caption{\tsp: Average CPU solving time in seconds for instances that were solved to optimality (number of instances that were solved to optimality).}\label{tab:tsp4}
\end{center}
\end{table}

The results presented in Table~\ref{tab:tsp4} show that \algls produces instances that are significantly harder than randomly generated instances.
When increasing the number of nodes to 16, it is the first time that an instance generated by \algls is unable to be solved within one hour.
Further increasing the number of nodes to 18 results in a significant increase in the run times to solve the hard robust optimization instances---only two instances can be solved within one hour.
These results further demonstrate the ability of \algls to produce hard instances for large robust optimization problems.

The relative increase in difficulty for the \algls instances over the \algru instance is presented in Table \ref{tab:tsp5}.
While these numbers are lower than for our previous methods for $n\le196$ (compare to Table~\ref{tab:tsp2}), the heuristic approach is shown to scale better for larger instances sizes.
It can be observed that \algls can produce instances that are a factor of up to 500 times more difficult to solve that the original randomly generated instance.
Since with $n\ge256$ many of the generated instances could not be solved to optimality, this decreases the factor increase.
Thus, it is expected that \algls is able to produce instances that are more than a factor of 500 times more difficulty than the \algru instances as the problem size increases.

\begin{table}[htbp]
\begin{center}
\begin{tabular}{r|rr}
$n$ & Avg & Max \\
\hline
100 & 3.1 & 11.1 \\
144 & 4.9 & 40.7 \\
196 & 14.8 & 151.8 \\
256 & $>41.8$ & $>504.3$ \\
324 & $>15.6$ & $>267.6$ \\
400 & $>2.7$ & $>20.5$
\end{tabular}
\caption{\tsp: Average and maximum time increase for each instance using \algls. When not solved to optimality, counted as time limit.}\label{tab:tsp5}
\end{center}
\end{table}

To finally highlight the potential of the \algls method, the instance generation times are presented in Table~\ref{tab:tsp6}.
For these results, the generation time limit was greatly restricted to 600 seconds.
Comparing Tables \ref{tab:tsp3} and \ref{tab:tsp6} it is clear that \algls is capable of generating harder instance than \algheu in much shorter run times.
This is a significant benefit to the approach heuristic approach for solving the \mro, since it enables the use of this technique to generate hard instances for difficult robust optimization problems.

\begin{table}[htbp]
\begin{center}
\begin{tabular}{r|rrrrrr}
$n$ & 100 & 144 & 196 & 256 & 324 & 400 \\
\hline
Time &0.1 & 24.4 & 392.7 & 504.2 & 552.2 & 570.2
\end{tabular}
\caption{\tsp: Average time to generate instances using \algls(600s limit).}\label{tab:tsp6}
\end{center}
\end{table}

\section{Conclusions}\label{sec:conclusions}

All relevant min-max robust combinatorial optimization problems are known to be NP-hard. But this theoretical complexity class does not necessarily indicate practical hardness. Indeed, randomly generated instances are usually not too challenging to solve with current off-the-shelf MIP solvers. Furthermore, algorithmic papers need to re-create such random instances every time, resulting not only in additional work for computational experiments, but also in another source of errors and incomparability of results.

The aim of this paper is to address these problems by introducing a way to generate problem instances which are considerably harder to solve than random instances. We first present exact and heuristic approaches for solving an optimization problem to generate hard robust problem instances. While these approaches require much computational effort to find hard instances, they are effective at increasing the difficulty over the randomly generated instances. To address the computational issues in generating hard robust instances, we propose the use of heuristic methods to solve the inner robust optimization problem within iterative solution algorithms. This latter approach is demonstrated to produce the most difficult instances within very short generation times. 

Our methods have been illustrated using the \selection problem and \tsp as examples, but they are widely applicable to other combinatorial problems. In further work, the foundation that is laid through these instances will be extended to a more comprehensive online problem library for robust optimization. All current instances are already available under \url{www.robust-optimization.com}. Also, the code used to generate the instances has been made available and can be easily extended to other underlying mathematical programming problems.

Three more approaches were tested, but not described in this paper. The first approach is based on constructing instances where the midpoint heuristic, a cornerstone of exact solution methods, performs badly. While an increase in hardness could be observed, results were not as promising as for \mro. Details on this setting and some computational results are provided in Appendix~\ref{sec:midpoint}. The second approach was based on the linear program developed in \cite{goerigk2018constructing} to construct a solution with small approximation guarantee for a given problem instance. We constructed instances by maximizing this approximation guarantee. The third approach was to train a neural network on a set of problem instance-solution time observations to predict hard instances. However, both approaches were not able to produce instances that were significantly harder than random instances. However, machine learning models seem to be a promising avenue for future research once hard instances have been generated by other models and thus become available for the training data.

\section*{Acknowledgments}

Stephen J. Maher is supported by the Engineering and Physical Sciences Research Council (EPSRC) grant EP/P003060/1.

\newcommand{\etalchar}[1]{$^{#1}$}

\newpage
\appendix

\section{Complexity of \mro}\label{app:complexity}

\begin{theorem}
The problem
\[ \max_{(\pmb{c}^1,\ldots,\pmb{c}^N)\in\cU} \min_{\pmb{x}\in\X} \max_{j\in[N]} \pmb{c}^j\pmb{x} \tag{\mro} \]
with a polyhedron $\cU\subseteq\mathbb{R}^{n \times N}$ is $\Sigma^p_2$-complete.
\end{theorem}
\begin{proof}
We first note that (\mro) is in class $\Sigma^p_2$, as for fixed $(\pmb{c}^1,\ldots,\pmb{c}^N)$, the resulting robust optimization problem is in NP.
We use a reduction from \textsc{2-Quantified 3-Dnf-Sat}, which is known to be $\Sigma^p_2$-complete \cite{stockmeyer1976polynomial}: Given a Boolean formula $\phi(\pmb{\alpha},\pmb{\beta})$ in DNF, where every clause consists of exactly three literals, is there a value for $\pmb{\alpha}=(\alpha_1,\ldots,\alpha_s)\in\{0,1\}^s$ such that for all $\pmb{\beta}=(\beta_1,\ldots,\beta_t)\in\{0,1\}^t$, the formula $\phi(\pmb{\alpha},\pmb{\beta})$ is true?

We build an instance of (\mro) using the \textsc{Representative Selection} problem \cite{kasperski2016robust} to define the underlying combinatorial problem. In this problem, we are given a partition of the items $\cup_{\ell\in[k]} T_\ell = [n]$ and a cost vector $\pmb{c}\in\mathbb{R}^n$. The task is to choose exactly one item of each $T_\ell$, such that the overall costs are minimized.

Our problem instance is built in the following way. 
Let $x_i\in\{0,1\}$ be the variable corresponding to variable $\alpha_i$, and let $y^1_i,y^2_i\in\{0,1\},$ correspond to $\beta_i$. Note that $n=s+2t$. We partition the items through sets $T_\ell$ for each $\ell\in[s]$, where $T_\ell=\{x_\ell\}$ (i.e., each $x_i$ must be equal to one), and $T'_\ell$ for each $\ell\in[t]$ with $T'_\ell=\{y^1_\ell,y^2_\ell\}$ (i.e, exactly one of $y^1_\ell$ and $y^2_\ell$ must be equal to one). Let us denote a clause of $\phi$ by
\[C_i=(a^i_1\alpha_1 \wedge a^i_2\alpha_2 \wedge \ldots \wedge a^i_s\alpha_s \wedge b^i_1\beta_1 \wedge b^i_2\beta_i \wedge \ldots \wedge b^i_t\beta_t) \] 
where $a^i_k$,$b^i_k$ denote the signs of the variables in $\{-1,0,1\}$ (exactly three signs are non-zero). Let $N$ clauses be given. We build a scenario for each clause. The corresponding polyhedron of possible scenarios is given as:
\begin{align*}
\cU = \Big\{ (\pmb{c}^1,\ldots,\pmb{c}^N)\ :\ & c^i(x_j) = a^i_j d_j & \forall i\in[N],j\in[s] \\
& d_j \in [-1,1] & \forall j\in[n]\\
& c^i(y^1_j) = b^i_j & \forall i\in[N],j\in[t] \\
& c^i(y^2_j) = -b^i_j & \forall i\in[N],j\in[t]\ \Big\}\hspace*{-3.5mm}
\end{align*}
where
\[ (\pmb{c}^1,\ldots,\pmb{c}^N) = 
\begin{pmatrix}
c^1(x_1) & c^2(x_1) & \ldots & c^N(x_1) \\
c^1(x_2) & c^2(x_2) & \ldots & c^N(x_2) \\
\vdots & \vdots & \ddots & \vdots \\
c^1(x_s) & c^2(x_s) & \ldots & c^N(x_s) \\
c^1(y^1_1) & c^2(y^1_1) & \ldots & c^N(y^1_1) \\
c^1(y^1_2) & c^2(y^1_2) & \ldots & c^N(y^1_2) \\
\vdots & \vdots & \ddots & \vdots \\
c^1(y^1_t) & c^2(y^2_t) & \ldots & c^N(y^2_t) \\
c^1(y^2_1) & c^2(y^2_1) & \ldots & c^N(y^2_1) \\
c^1(y^2_2) & c^2(y^2_2) & \ldots & c^N(y^2_2) \\
\vdots & \vdots & \ddots & \vdots \\
c^1(y^2_t) & c^2(y^2_t) & \ldots & c^N(y^2_t) \\
\end{pmatrix} \]
Note that the projection of $\cU$ onto $\mathbb{R}^{n \times N}$ is indeed a polyhedron.
We claim that there exists an optimal solution to our MRO instance with objective value at least 3 if and only if the \textsc{2-Quantified 3-Dnf-Sat} instance is true. Let us first assume that there exists $\pmb{\alpha}$ such that $\phi(\pmb{\alpha},\pmb{\beta})$ is true for all $\pmb{\beta}$. We construct $N$ scenarios by setting $d_i=1$ if $\alpha_i$ is true, and $d_i=-1$ otherwise. Then the resulting robust problem becomes
\begin{align*}
\min\ & z \\
\text{s.t. } & z \ge \sum_{j\in[s]} a^i_jd_j + \sum_{j\in[t]} b^i_j(y^1_j-y^2_j) & \forall i\in[N] \\
& y^1_j + y^2_j = 1 & \forall j\in[t] \\
& y^1_j,y^2_j\in\{0,1\} & \forall j\in[t]
\end{align*}
By construction, it follows that for every feasible solution $(\pmb{y}^1,\pmb{y}^2)$, the optimal value of $z$ is 3. Let us now assume that the objective value of the (\mro) is at least 3. Note that in this case, the objective value is exactly three, and we can assume that $d_j\in\{-1,1\}$ for all $i\in[s]$. Set $\alpha_j$ as true if and only if $d_j=1$. Then, it follows that for every possible value $\pmb{\beta}$, the formula $\phi$ is true, as the robust problem aims at finding values for $(\pmb{y}^1,\pmb{y}^2)$ such that all clauses are false.
\end{proof}

\section{Maximizing the Midpoint Objective Value}\label{sec:midpoint}

We now explore a different view on problem hardness. Instead of maximizing the objective value of the resulting optimal solution, which, as the discussion in Section~\ref{sec:1} has shown, is a complex optimization problem, we use the objective value of the midpoint solution as a proxy. The midpoint method is one of the most popular heuristics for min-max robust combinatorial optimization. It aggregates all scenarios into one average scenario and solves the resulting single-scenario problem, which is possible in polynomial time for some combinatorial problems (see Assumption~\ref{lpassumption}). It is known to give an $N$-approximation to the robust problem \cite{Aissi2009}, and has been the best known general method until recently \cite{chassein2017scenario}. Due to its simplicity, it is also a popular submethod for exact branch-and-bound approaches \cite{chassein2015new}.

The optimization problem to generate hard instances we consider here is therefore given as
\begin{equation}
  \max_{\pmb{c}^1,\ldots,\pmb{c}^N} \max_{i\in[N]} \pmb{c}^i \hat{\pmb{x}}\left(\frac{1}{N}\sum_{\ell\in[N]} \pmb{c}^\ell\right) \tag{\midpt}
  \label{eqn:maximisingMidpoint}
\end{equation}
where $\hat{\pmb{x}}(\pmb{c})$ denotes an optimal solution to scenario $\pmb{c}$.

To enable the use of general purpose mixed integer programming solvers, a reformulation of problem \eqref{eqn:maximisingMidpoint} is performed.
A common reformulation involves applying a linearization if the nominal problem can be written as a linear program under Assumption~\ref{lpassumption}, which is the case for \selection.
In the following, we present a reformulation of \eqref{eqn:maximisingMidpoint} when \selection is the nominal robust optimization problem.
To apply this linearization, we enforce that $\pmb{x}$ is an optimal solution to the midpoint scenario $\frac{1}{N}\sum_{\ell\in[N]} \pmb{c}^\ell$ by requiring the corresponding primal and dual objective values to be equal.
The resulting optimization problem is then
\begin{align}
\max\ &\sum_{i\in[N]} t_i \lambda_i \label{mcon1}\\
\text{s.t. } & t_i = \sum_{k\in[n]} c^i_k x_k & \forall i\in[N] \label{mcon2}\\
& \sum_{i\in[N]} \lambda_i = 1 \label{mcon3}\\
& \sum_{i\in[N]} \sum_{k\in[n]} c^i_k x_k = p\alpha - \sum_{k\in[n]} \beta_k \label{mcon4}\\
& \sum_{k\in[n]} x_k = p \label{mcon5}\\
& \alpha - \beta_k \le \sum_{i\in[N]} c^i_k & \forall k\in[n] \label{mcon6}\\
& \lambda_i \in\{0,1\} & \forall i\in[N] \label{mcon7}\\
& \pmb{c}^i \in\cU^i & \forall i\in[N] \label{mcon8}\\
& t_i \ge 0 & \forall i\in[N] \label{mcon9}\\
& x_k \in\{0,1\} & \forall k\in[n] \label{mcon10}\\
& \alpha \ge 0 \label{mcon11}\\
& \beta_k \ge 0 & \forall k\in[n] \label{mcon12}
\end{align}
Here, $t_i$ denotes the objective value of the midpoint solution in scenario $\pmb{c}^i$ (see Constraint~\eqref{mcon2}). The optimization problem maximizes the largest $t_i$ by choice variables $\lambda_i$ (see Objective~\eqref{mcon1} and Constraint~\eqref{mcon3}). Constraints~(\ref{mcon4}-\ref{mcon6}) ensure that $\pmb{x}$ is indeed the midpoint solution by enforcing primal and dual feasibility, and equality of primal and dual objective values. 

There are still nonlinearities between $t_i \lambda_i$ and $c^i_kx_k$. We linearize the first product using $q_i = t_i \lambda_i$ with $q_i \le t_i$ and $q_i \le M_i \lambda_i$, where $M_i=\sum_{k\in[n]} \overline{c}^i_k$ suffices. The second product is linearized using $r_{ik} = c^i_kx_k$ with $r_{ik} \le c^i_k$ and $r_{ik} \le \overline{c}^i_kx_k$.

We now compare this approach to \algex with a similar experimental setup as before.
For ease of exposition, the evaluation of the efficacy of the exact solution methods will be performed using \selection only. We refer to the results using the midpoint method as \algmid.

The average run times of the instances generated from \algex and \algmid are presented in Table~\ref{tab:exact-cpu-solve}. For comparison, the average run times to solve the randomly generated instances is also presented.

\begin{table}[htb]
\begin{center}
\begin{tabular}{rr|rrr}
& & \multicolumn{3}{c}{$n=N$} \\
 Budget & Method & 20 & 30 & 40 \\
 \hline
 \multirow{2}{*}{1} & \algex & 0.04 & 0.61 & 7.83 \\
 & \algmid & 0.03 & 0.15 & 1.62 \\
 \hline
 \multirow{2}{*}{2} & \algex & 0.06 & 0.92 & 7.49 \\
 & \algmid & 0.03 & 0.16 & 1.74 \\
 \hline
 \multirow{2}{*}{5} & \algex & 0.08 & 0.77 & 10.59 \\
 & \algmid & 0.03 & 0.23 & 2.90 \\
\hline
  & \algru & 0.03 & 0.13 & 1.43
\end{tabular}
\caption{Average CPU time in seconds when solving the random instances and the instances generated using the exact iterative method and the midpoint method.}\label{tab:exact-cpu-solve}
\end{center}
\end{table}

\begin{table}[b]
\begin{center}
\begin{tabular}{rr|rrr}
& & \multicolumn{3}{c}{$n=N$} \\
 Budget & Method & 20 & 30 & 40 \\
 \hline
 \multirow{2}{*}{1} & \algex & 1.7 & 207.0 & 3005.8 \\
  & \algmid & 1.1 & 15.2 & 49.2 \\ 
  \hline
 \multirow{2}{*}{2} & \algex & 49.3 & 3238.1 & 3796.0 \\
 & \algmid & 2.2 & 28.3 & 148.5 \\
 \hline
 \multirow{2}{*}{5} & \algex & 3677.9 & 4305.1 & 4081.9 \\
 & \algmid & 6.0 & 1047.4 & 1807.2 \\ 
\end{tabular}
\caption{Average CPU time in seconds to produce instances using the exact iterative method and the midpoint method. A (lenient) time limit of 3600 seconds was used.}\label{tab:exact-cpu}
\end{center}
\end{table}

The results presented in Table \ref{tab:exact-cpu-solve} show that while \algmid produces instances that are more difficult than random instances, the increase in difficulty is less than that achieved by \algex.
Given that \algmid is a more complex algorithm for generating problem instances than a random generator, the results presented in Table \ref{tab:exact-cpu-solve} suggest that \algmid is not a satisfactory method for instance generation.
This is further highlighted by the average computation times of \algmid presented in Table \ref{tab:exact-cpu}.
These results demonstrate that maximizing the minimum solution objective is a better proxy for instance \emph{hardness} than maximizing the midpoint objective.

\end{document}